\documentclass[
submission
]{dmtcs-episciences}

\usepackage[utf8]{inputenc}

\usepackage{subfigure}

\usepackage{hyperref}

    \usepackage{amsmath,amsthm,amssymb,mathtools}
    \usepackage{enumerate}
    \usepackage{cleveref} 
    \usepackage{color}

\usepackage[round]{natbib}

\def\Z{\mathbb Z}

\def\N{\mathbb N}
\def\A{\mathcal A}
\def\B{\mathcal B}
\def\C{\mathcal C}

\def\P{\mathcal P}

\def\S{\mathcal S}

\def\P{\mathcal P}

\def\L{\mathcal L}
\def\Lu{{\mathcal L}(\uu)}
\def\Lv{{\mathcal L}(\vv)}
\def\uu{\mathbf u}
\def\vv{\mathbf v}
\def\tt{\mathbf t}
\def\ww{\mathbf w}

\def\Pal{{\rm Pal}}
\def \id {{\rm Id}}

\def\SS{S}

\def \Rk#1 {$\mathcal{R}_{#1}$}
\def \Rext {{\rm Rext}}
\def \Lext {{\rm Lext}}

\def \Pext {{\rm Pext}}

\def \b {{\rm b}}

\newtheorem{thm}{Theorem}
\newtheorem{theorem}[thm]{Theorem}

\newtheorem{corollary}[thm]{Corollary}

\newtheorem{lemma}[thm]{Lemma}

\newtheorem{proposition}[thm]{Proposition}

\theoremstyle{definition}

\newtheorem{definition}[thm]{Definition}
\newtheorem{example}[thm]{Example}

\crefname{thm}{theorem}{theorems}
\crefname{thrm}{theorem}{theorems}
\crefname{coro}{corollary}{corollaries}
\crefname{example}{example}{examples}
\crefname{lem}{lemma}{lemmas}
\crefname{lmm}{lemma}{lemmas}
\crefname{claim}{claim}{claims}
\crefname{obs}{observation}{observations}
\crefname{proposition}{proposition}{propositions}
\crefname{prop}{proposition}{propositions}
\crefname{defi}{definition}{definitions}

\crefname{theorem}{theorem}{theorems}
\crefname{corollary}{corollary}{corollaries}
\crefname{example}{example}{examples}
\crefname{lemma}{lemma}{lemmas}
\crefname{proposition}{proposition}{propositions}
\crefname{definition}{definition}{definitions}

\theoremstyle{remark}
\newtheorem{remark}[thm]{Remark}

\crefname{example}{example}{examples}

\author{Edita Pelantov\'a\affiliationmark{1} 
  \and \v St\v ep\'an Starosta\affiliationmark{2} 
}
\title[Constructions of words rich in palindromes and pseudopalindromes]{Constructions of words rich \\ in palindromes and pseudopalindromes}
\affiliation{
  Faculty of Nuclear Sciences and Physical Engineering, Czech Technical University in Prague, Czech Republic\\
  Faculty of Information Technology, Czech Technical University in Prague, Czech Republic
}
\keywords{palindrome, palindromic defect, rich words, full words, Rote words}
\received{2015-6-12}

\accepted{2016-10-18}

\revised{2016-6-20}
\begin{document}
\publicationdetails{18}{2016}{3}{16}{1513}
\maketitle
\begin{abstract}
A  narrow connection between infinite binary  words rich in classical palindromes and infinite  binary  words  rich simultaneously in palindromes and pseudopalindromes (the so-called $H$-rich words) is  demonstrated.
 The correspondence between  rich and  $H$-rich words is based on the operation $S$  acting over words over the alphabet $\{0,1\}$ and  defined by   $S(u_0u_1u_2\ldots) = v_1v_2v_3\ldots$, where  $v_i= u_{i-1} + u_i  \mod 2$.
 The operation $S$  enables us to construct  a new class of rich words and a new class of $H$-rich words.
 Finally, the operation $S$ is considered on the multiliteral alphabet $\mathbb{Z}_m$ as well and applied to the generalized Thue--Morse words. As a byproduct,  new   binary rich  and $H$-rich words are obtained  by application of $S$ on the generalized Thue--Morse words over the alphabet $\mathbb{Z}_4$.
\end{abstract}






\section{Introduction}

In the present paper we concentrate on construction of infinite words which are filled with palindromes  or  pseudopalindromes to the highest possible level, the so-called \emph{rich} words.
Before we explain the expression ``the highest possible level'' we recall the basic notions we work with.
We understand by an \emph{infinite word} over a finite alphabet $\A$ a sequence $\uu =
(u_n)_{n\in \mathbb{N}} = u_0u_1u_2\ldots$, where $u_n \in \A$ for each $n \in \N$. A \emph{factor} of $\uu$ is a finite sequence $w =w_0w_1\cdots w_{n-1}$ of letters from $\A$ such that  $w=u_iu_{i+1}\cdots u_{i+n-1 }$ for some  $i,n \in \N$.  The set of all factors of $\uu$ is the \emph{language} of $\uu$, usually denoted $\Lu$.
A finite word $w=w_0w_1\cdots w_{n-1}$ is called palindrome if $w$ coincides with  its reversal $R(w) = w_{n-1}w_{n-2}\cdots w_1w_0$.

Infinite words whose language contains infinitely many palindromes
are being studied by many authors. Apart from the impulses from
outside mathematics (such as   \cite{HoKnSi} where these words are
used in a model of solid materials with finite local complexity) the
main reason of the interest of mathematicians is the variety of
characterizations of rich  words.   To specify the expression ``the  highest possible level'' one can adopt two distinct points of view: local and global.

From the local point of view, one looks at a finite piece of the infinite word, i.e., at a factor of $\uu$,   and counts the number of distinct palindromes occurring in this factor.
A motivation for rich word definition was an inequality due to
Droubay and Pirillo, see \cite{DrPi}, which states that a finite word of
length $n$ contains at most $n+1$ distinct palindromes (the empty
word is counted as a palindrome).
 An infinite word is \emph{rich}, or
\emph{full}, if  every its factor of length $n$ contains $n+1$ distinct palindromes.

From the global point of view, one counts the palindromes of length $n$  in the set of all factors of  $\uu$, i.e., in the language $\Lu$.
Let $ \C_{\uu}(n)$ and $ \P_{\uu}(n)$ denote the number of  factors of length $n$ and the number of palindromic factors of length $n$, respectively.   As shown in \cite{BaMaPe},  if $\Lu$  is closed under reversal, then the number of palindromes in $\uu$ is bounded from above by the relation
\begin{equation} \label{eq:BaMaPe}
\C_{\uu}(n+1) - \C_{\uu}(n) + 2 \geq \P_{\uu}(n+1) + \P_{\uu}(n) \quad \text{ for every } n \in \N.
\end{equation}
In \cite{BuLuGlZa}, Bucci, De Luca, Glen and Zamboni show that for infinite words with language
closed under reversal  the local and global points of view coincide. More precisely, $\uu$ is rich if and only if the
inequality in \eqref{eq:BaMaPe} can be written as an equality for every $n \in \N$.

Classic examples of rich words on binary alphabets include Sturmian
words, i.e.,  infinite words over binary alphabet with the factor complexity $\C_{\uu}(n)=n+1$ for each $n \in \N$. Sturmian words can be generalized to multiliteral alphabets
in many ways, see for example \cite{BaPeSta2}. Two of these generalizations,
namely $k$-ary Arnoux--Rauzy words and words coding $k$-interval
exchange transformation with symmetric interval permutation, are
rich as well. Both mentioned classes have their language closed
under reversal.

Blondin Mass\'e, Brlek, Garon and Labb\'e  showed in \cite{BlBrLaVu11}  that rich words include
comple\-mentary-symmetric Rote words. They can be defined as binary
words with factor complexity $\C_{\uu}(n) = 2n$ for every nonzero
integer $n$ and with language closed under the exchange of
letters, see \cite{Rote}. This implies that the language of a
complementary-symmetric Rote word is closed under two mappings
 acting on the set $\{0,1\}^*$  of all finite binary words: the first is $R$ and the second is $E$ defined by
$E(w_0 \cdots w_n) = E(w_n) \cdots E(w_0)$ for letters $w_i$ and
$E(0) = 1$ and $E(1) = 0$. Thus, the language of a complementary-symmetric Rote word is closed under all
elements of a group $H = \{ R, E, ER, \id \}$.  The same property has the  language of the famous  Thue--Morse word  $\tt$, nevertheless, it is well-known that $\tt$  is not rich.

For binary words having language closed under all
elements of $H$, we show     in  \cite{PeSta1}  that
\begin{equation} \label{eq:NerovnostH}
\C_{\uu}(n\!\!+\!\!1) - \C_{\uu}(n) + 4\geq \P_{\uu}(n\!\! +\!\!1) + \P_{\uu}(n) + \P_{\uu}^E(n\!\!+\!\! 1) + \P_{\uu}^E(n) \quad \text{ for every } n \geq 1, \qquad
\end{equation}
where $\P_\uu^E$ is  the function counting
 $E$-palindromes -- words fixed by
$E$ -- in the word $\uu$. Analogously to the case of equality in \eqref{eq:BaMaPe}, we
say that an infinite word with language closed under all elements
of $H$ is \emph{$H$-rich} if  in \eqref{eq:NerovnostH} the equality holds  for
all $n \geq 1$.  We also demonstrated that  the Thue--Morse word $\tt$ is $H$-rich. In \cite{Sta2011} the second author proved   that the  binary generalization  $\tt_{b,2}$ of the Thue--Morse word is   $H$-rich for all $b\geq 2$ (the definition of  $\tt_{b,2}$
 is recalled  in Preliminaries).  In fact, the words $\tt_{b,2}$ are the only $H$-rich words that have been found up to now.

One of the main aims of the present article is to  describe  a procedure which  produces  new $H$-rich words.
We have found an inspiration in a connection between complementary-symmetric Rote words and Sturmian words due to Rote in \cite{Rote}.
Given an infinite word $\uu = u_0u_1 \ldots \in \{0,1\}^\N$, we set $\SS(\uu) = v_1v_2 \ldots \in
\{0,1\}^\N$ with $v_i = (u_{i-1} + u_i) \bmod{2}$ for all positive integer $i$.
The operator $\SS$ defines the mentioned relation: a word
$\uu$ is a  complementary-symmetric Rote word if and only if
$\SS(\uu)$ is a Sturmian word.

In Section \ref{sec:binary},  we investigate binary words which are simultaneously $H$-rich  and also rich in the classical sense. In particular,  we prove  that every complementary-symmetric Rote word is  $H$-rich, see Corollary  \ref{RoteHrich}.  The main result concerning $H$-richness is presented in Theorem  \ref{thm:bin1}.  On the  one hand the theorem says  that the operator $\SS$ applied to  an
$H$-rich  word produces a rich word. Using the examples of
$H$-rich words mentioned earlier, we get a new class of  rich words, namely the words
$\SS(\tt_{b,2})$  for all $b \geq 2$. On the other hand, the theorem  transforms the task to discover new $H$-rich words to the task to
discover   a new class of rich words with special structures of palindromes.
One such class is described in \cite{Sta2015}.
Section  \ref{sec:multi} is devoted to the notion of $G$-richness on a multiliteral alphabet.  In particular, the operation $S$ is defined over the alphabet $\Z_m$.  Theorem \ref{theoremH}   illustrates that even on a multiliteral alphabet the operation $S$ connects $G$-richness  and $G'$-richness for, in general, distinct groups $G$ and $G'$. In this sense  Theorem \ref{theoremH} is a  weaker version of  Theorem  \ref{thm:bin1}.


\section{Preliminaries} \label{sec:prelim}

The set $\A^*$ is the set of all finite words over the
\emph{alphabet} $\A$ which is a finite set of \emph{letters}.
The \emph{length} of the word $w= w_0w_1\cdots w_{n-1} \in \A^*$ with $w_i \in \A$  for all $i$ is denoted $|w|$ and equals $n$.
 The \emph{empty word} -- the unique word of
length $0$ -- is denoted $\varepsilon$. The set $\A^*$ together
with concatenation forms a free monoid with the neutral element
$\varepsilon$.
A word $v\in \A^*$ is a \emph{factor} of $w\in \A^*$ if $w= uvz$ for some word $u,z\in \A^*$.
 If, moreover, $u =\varepsilon$, then we say that  $v$ is  a \emph{prefix} of $w$,   if $z =\varepsilon$, the word $v$ is a \emph{suffix} of $w$.
If $w$ has the form  $w=vz$, then $z$  is denoted  $z=v^{-1}w $  and the word $v^{-1}w v$ is a \emph{conjugate} of the word $w$.

The \emph{infinite word} over  $\A$ is a sequence $\uu =
(u_n)_{n\in \mathbb{N}} = u_0u_1u_2\ldots$. The symbol $\A^\N$ denotes the set of
all infinite words over  $\A$. A finite word $w \in
\A^*$   of length $n =|w|$  is a \emph{factor} of $\uu$  if
there exists an index $i$ such that $w=u_iu_{i+1} \cdots
u_{i+n-1}$; the index $i$ is an \emph{occurrence} of the
factor $w$. The symbol  $\L_n(\uu)$ stands for the set of all
factors of length $n$ occurring in $\uu$. The set of all factors  of
$\uu$ is the \emph{language} of $\uu$ and is denoted by $\L(\uu)$.

An infinite word $\uu$ is \emph{recurrent} if any factor  of $\uu$ has at least two occurrences in $\uu$.
Equivalently, a word is recurrent if any factor has infinitely many occurrences.
If moreover for any factor $w$ the gaps between consecutive occurrences of $w$ are bounded,
then the word $\uu$ is \emph{uniformly recurrent}.
Let $w$ and $vw$ be  factors of $\L(\uu)$ such that $vw$ has a prefix $w$ and $w$ occurs in  $vw$  exactly twice.
The word $v$ is a \emph{return word} of $w$  and $vw$ is a \emph{complete return word} of $w$.
One can say equivalently:  a recurrent word $\uu$  is uniformly recurrent if any factor $w \in \L(\uu)$ has finite number of return words of $w$.

The \emph{factor complexity} of $\uu$ is the mapping $\C_{\uu}:
\mathbb{N} \to \mathbb{N}$,  defined by $\C_\uu(n) = \#
\L_n(\uu)$.  Given $a \in \A$ and $w \in \A^*$, a factor $wa \in  \L(\uu)$  is a \emph{right
extension} of the factor $w$. Any factor of $\uu$ has at least
one right extension, the set of all right extensions of
$w$  is denoted $\Rext(w)$.
If $w$ has at least two right
extensions we call it \emph{right special}. Analogously one can
define \emph{left extension} and \emph{left special} and $\Lext(w)$. In a recurrent word $\uu$ any factor has at least one
left extension.
A factor $w$ which is left and right special is \emph{bispecial}.
Special factors can be used to determine the factor complexity, in particular
\begin{equation*}\label{1diff}
\Delta \C_\uu(n) = \C_\uu(n+1) - \C_\uu(n) = \sum_{w \in
\L_n(\uu)} \bigl(\# \Rext(w)-1\bigr).
\end{equation*}
If $\A$ is a binary alphabet, we get
\begin{equation}\label{1diffBin}
\Delta \C_\uu(n) = \#\{w \in \L_n(\uu) \colon w \hbox{ is right
special}\}.
\end{equation}

 A mapping $\mu: \A^* \to
\B^*$ is a \emph{morphism} if $\mu(wv) = \mu(w)\mu(v)$ for all
$w,v \in \A^*$. It is an \emph{antimorphism} if $\mu(wv) =
\mu(v)\mu(w)$ for all $w,v \in \A^*$. An infinite word $\uu$ is
\emph{closed under the mapping $\mu$} if $\mu(w) \in \L(\uu)$ for
any factor $w \in \L(\uu)$.  Domain of a morphism $\varphi:\A^* \to
\A^*$    can be naturally extended to $\A^\N$ by the prescription
$\varphi(\uu) = \varphi(u_0u_1u_2\ldots) = \varphi(u_0)  \varphi(u_1)\varphi(u_2)\ldots.$
An infinite word $\uu \in  \A^\N$ is called  \emph{fixed point} of  a morphism $\varphi$ if $\varphi(\uu) = \uu$.

An antimorphism $\Psi$ is \emph{involutory} if $\Psi^2=\id$.  The
most frequent involutory antimorphism is the reversal mapping $R$.
If the word $\uu$ is closed under an involutory antimorphism, then  $\uu$ is
necessarily recurrent.

If $p = \Psi(p)$, the word $p$ is a \emph{$\Psi$-palindrome} or \emph{pseudopalindrome}, if specification of the mapping $\Psi$ is not needed.
In the case $\Psi=R$, we say only palindrome instead of
$R$-palindrome. The set of all $\Psi$-palindromes occurring as factors of
a finite word $w$ is denoted $\Pal^\Psi(w)$.
The \emph{$\Psi$-palindromic complexity} of an infinite word $\uu$ is
the mapping $P^\Psi_{\uu}: \mathbb{N} \to \mathbb{N}$, defined by
$\P^\Psi_\uu(n) = \# \{p \in \L_n(\uu)\colon p= \Psi(p)\}$.

A $\Psi$-palindrome $w$ is \emph{centered at $x \in \A \cup \{\varepsilon\}$} if $w = vx\Psi(v)$ for some word $v$.
If a $\Psi$-palindrome is centered at $\varepsilon$, then it is of even length.

\section{$G$-defect and $G$-richness}\label{DefRichness}
First, we recall the definition of  palindromic  defect as it was introduced by Brlek,  Hamel, Nivat and Reutenauer  in \cite{BrHaNiRe}.
This classical definition is based on the
inequality
\begin{equation}\label{Rbound}
\#\Pal^R(w)\leq |w|+1 \quad \quad \hbox{for all } w \in \A^*,
\end{equation}
where $ \Pal^R(w)$ is the set of all $R$-palindromic factors of $w$ including the empty word.

The \emph{$R$-defect} of a finite word $w$ is
\begin{equation*}\label{RdefectFin} D^R(w) = |w|+1  -
\#\Pal^R(w),
\end{equation*}
and $R$-defect of an infinite word $\uu$  is
\begin{equation*}\label{RdefectInfin} D^R(\uu)= \sup\{D^R(w) \colon  w \in \L(\uu)\}.
\end{equation*}
We prefer to use the name $R$-defect instead of the originally used ``defect''  because  we will introduce an analogous notion for a general antimorphism $\Psi$ as well.
An infinite word $\uu$ with $D^R(\uu)=0$ is called \emph{$R$-full} or
\emph{$R$-rich}. If $D^R(\uu)$ is finite,  we say that $\uu$ is
\emph{almost $R$-rich}.
In \cite{BrRe-conjecture},  the inequality  \eqref{eq:BaMaPe} is used to introduce the value
\[
T_{\uu}(n) = \Delta \C_\uu(n) +2 - \P_{\uu}^R(n+1) - \P_{\uu}^R(n) \quad \quad
\text{ for every } n \in \N
\]
 and they conjectured  that if $\uu$ is closed
 under reversal, then
\begin{equation}\label{DefectFormula}
2D^R(\uu) = \sum_{n=1}^\infty T_{\uu}(n).
\end{equation}
Their conjecture was proven  in \cite{BaPeSta5}. In particular,
it means that $D^R(\uu)$ is finite if and only if  there exists
$N\in \mathbb{N}$ such that $T_{\uu}(n) = 0$ for all $n\geq N$, or
in other words in \eqref{eq:BaMaPe}
 the equality holds for all $n\geq N$.

To prove $R$-richness  we will use the
characterization of $R$-rich words given in \cite{BaPeSta}. It
exploits the notion of the bilateral order $\b(w)$  of a factor $w$ and
the palindromic extension of a~palindrome. The bilateral order was
introduced in \cite{Ca} as
\begin{equation}\label{eq:bilateralorder}
\b(w)   =  \# \{ awb \in \Lu \colon a,b \in \A \} - \# \Rext(w) - \#\Lext(w)+1.
\end{equation}
 The set of all palindromic extensions of  a
palindrome $w \in \mathcal{L}({\mathbf u})$ is defined by
\[
\Pext(w) =\{ awa \colon awa \in \mathcal{L}({\mathbf u}), a \in
\mathcal{A} \}.
\]

\begin{theorem}[\cite{BaPeSta2}] \label{thm:minus1}
Let $\uu$ be an infinite word closed under
reversal. \begin{enumerate}
\item   The word $\uu$ is $R$-rich  if and only if any bispecial factor
$w$ of $\uu$ satisfies:
\begin{equation}\label{eq:BilateralProRich}
\b(w) = \begin{cases} \# \Pext(w) - 1 & \text{ if $w$ is a
palindrome;} \\ 0 & \text{ otherwise.} \end{cases}
\end{equation}
\item If  the  word $\uu$ is almost $R$-rich, then \eqref{eq:BilateralProRich} is satisfied for all bispecial factors $w$ up to finitely many exceptions.
\end{enumerate}
\end{theorem}

The first attempt to study the number of $\Psi$-palindromes for an
involutory antimorphism $\Psi$ was made in \cite{BlBrGaLa}. Blondin Mass\'e, Brlek, Garon and Labb\'e
considered the binary alphabet $\{0,1\}$ and the antimorphism  $E$.
They showed that
\begin{equation}\label{Ebound}
\#\Pal^E(w)\leq |w|  \quad \hbox{for all } w \in \A^* \setminus \{\varepsilon\}.
\end{equation}
In \cite{Sta2010}, this results is generalized for an arbitrary
involutory antimorphism $\Psi$ and arbitrary alphabet into the inequality
\begin{equation}\label{Psibound}
\#\Pal^\Psi(w)\leq |w|+1 - \gamma_\Psi(w) \quad \hbox{for all } w
\in \A^* ,
\end{equation}
where $\gamma_\Psi(w) = \#\bigl\{ \{a,\Psi(a)\} \colon a \in \A,   \ a
\hbox{ occurs in $w$ and } \Psi(a)\neq a\bigr\}$. Clearly, if
$\Psi = E$ we have \eqref{Ebound} as $\gamma_E(w) = 1$ for any
$w\neq \varepsilon$, if $\Psi = R$ we have \eqref{Rbound} as
$\gamma_R(w) = 0$ for any $w$. Based on the inequality
\eqref{Psibound}, the \emph{$\Psi$-defect} of $w \in \A^*$ is defined by
\begin{equation}\label{PsidefectFin} D^\Psi(w) = |w|+1  -
\gamma_\Psi(w)-  \#\Pal^\Psi(w).
\end{equation}
The $\Psi$-defect of an infinite word $\uu$ is defined analogously, i.e.,  $D^\Psi(\uu) = \sup \{ D^\Psi(w) \colon w \in \Lu\}$.

Infinite words having finite $\Psi$-defect can be characterized
by several properties, for more details about $R$-defect see
\cite{BaPeSta3} and about $\Psi$-defect see \cite{Sta2010,PeSta_Milano_IJFCS}.  In \cite{PeSta_Milano_IJFCS} we showed that there exists a very narrow connection between words with finite defect  and words with zero defect. We
proved that if $\uu$ is closed under an involutory antimorphism
$\Psi$ and $D^\Psi(\uu)$ is finite, then $\uu$ is a morphic image of
a word $\vv$ with $D^\Phi(\vv) =0$ for some involutory
antimorphism $\Phi$. If moreover $\uu$ is uniformly recurrent, then
$\Phi = R$.   In this sense, considering $\Psi$ instead of $R$
does not bring a broader variability into the concept of rich words.

The situation changes when we consider more antimorphisms.  In
\cite{PeSta1} we defined a generalization of the notion of defect.
In what follows, the symbol $G$ stands for a finite group
consisting of morphisms and antimorphisms over $\A^*$ and
containing at least one antimorphism. The \emph{orbit} of $w \in
\A^*$ is the set
 \begin{equation}\label{orbit} [w] = \{ \mu(w) \colon \mu \in G \}.
\end{equation}
We say that $\uu$ is closed under $G$ if $[w] \subset \L(\uu)$ for
any $w \in \L(\uu)$. Word $p \in \A^*$ is a
\emph{$G$-palindrome} if $p=\Psi(p)$  for some antimorphism
$\Psi\in G$.
The generalization of the set of all palindromic factors of a word is a set consisting of palindromic orbits, namely the set
\[
\Pal^G(w) =  \bigl\{[p] \colon p \hbox{ occurs in $w$ and $p$ is a $G$-palindrome}     \bigr\}.
\]
Note that if $G= \{\id, \Psi\}$ where $\Psi$ is an involutory
antimorphism, then $\Pal^\Psi(w)$ is in one-to-one correspondence with the set $\Pal^G(w)$ (the only difference is that the latter is a set of orbits instead of factors).
Let us stress that
in $ \Pal^G(w)$ we count how many different orbits have a
$G$-palindromic representative occurring in $w$.

\begin{definition}\label{def:D_G_nove}
Let $w$ be a finite word.
The \emph{$G$-defect} of $w$ is defined as
\[
D^G(w) = |w| + 1 - \# \Pal^G(w) - \gamma_G(w),
\]
where
\[
\gamma_G(w) = \# \left \{ [a] \colon  a \in \A, a \text{ occurs in
} w \text{, and } a \neq \Psi(a)   \text{  for every
antimorphism } \Psi \in G \right \}.
\]
\end{definition}
A finite word is \emph{$G$-rich} if its $G$-defect is $0$. An
infinite word is $G$-rich  if all its factors are $G$-rich.
In \cite{PeSta1}, a distinct and equivalent definition of $G$-richness is used: it is based on a specific structure of graphs representing the factors of same length of the word.

\begin{example}\label{ThueMorsefinite}
We illustrate the previous notions on the Thue--Morse word $\tt$, the
fixed point of the morphism $0\mapsto 01$ and $1\mapsto 10$
starting with $0$, i.e., $\tt = 0110 1001 1001 0110 10 \cdots $.
 The word $\tt$ is closed
under $R$ and $E$. Let $H=\{\id, R,E,ER\}$. For the group
$H$ the value $\gamma_H(w) = 0$ for any $w \in \mathcal{A}^*$.
Consider
$w = 0110 1001 1001$, the prefix of $\tt$  of length $12$.
We have
\begin{align*}
\Pal^R(w) & =  \{ \varepsilon, 0,1,11,00,101,010,0110,1001,001100,10011001\}, \\
\Pal^E(w) & =  \{ \varepsilon, 01,10,0011,1100,1010,110100, 001100, 01101001\}, \\
\Pal^H(w) & =  \{ [\varepsilon],  [0],[00],[01],[010],[0110],[0011],[1010],[110100], \\ & \qquad [100110],[001100],[10011001],[01101001] \}.
\end{align*}
The corresponding defects of $w$ are
\begin{align*}
D^R(w) &= |w| +1 - \#\Pal^R(w) = 2, \\
D^E(w) &= |w|  - \#\Pal^E(w) = 3, \\
D^H(w) &= |w| +1 - \#\Pal^H(w) = 0.
\end{align*}
In fact, the Thue--Morse word is $H$-rich, whereas its
$R$-defect and $E$-defect are both infinite, see Example \ref{ThueMorse} later.
\end{example}

For  $G$-richness, theorems analogous to the theorems for the classical richness can be stated, c.f. \cite{PeSta2}.
The list of known $G$-rich words with $G$ having at least two antimorphisms  is modest.
It contains the generalized Thue--Morse words $\tt_{b,m}$.
The word $\tt_{b,m}$ is defined on
the alphabet $\{0, \ldots, m-1\}$ for all $b \geq 2$ and $m \geq
2$ as
\[
\tt_{b,m} = \left ( s_b(n) \mod m \right )_{n=0}^{+\infty},
\]
where $s_b(n)$ denotes the sum of digits in the base-$b$
representation of the integer $n$. See for instance \cite{AlSh,CUSICK20114738}
where this class of words is studied. The language  of $\tt_{b,m}$ is
closed under a group isomorphic to the dihedral group of order
$2m$,  here denoted $I_2(m)$.  In Section \ref{sec:multi}, we describe the group in details.
In  \cite{Sta2011}, the second author proved  that $\tt_{b,m}$ is $I_2(m)$-rich
for any parameters $b \geq 2$ and $m \geq 2$.

In \Cref{RoteHrich} we add to the list of $H$-rich words also complementary-symmetric Rote words.  As already mentioned in Introduction, an infinite binary word
 $\uu$ is a \emph{complementary-symmetric Rote word} if
its factor complexity satisfies $\C_\uu(n) = 2n$ for all $n \geq 1$ and its language is closed under the exchange of the two letters $E$.

In this article, we focus on groups $G$ acting on $\A^*$ for which the implication
\[
\Psi_1(a)= \Psi_2(a) \quad \Longrightarrow \quad  \Psi_1= \Psi_2
\]
is true for any letter $a \in \A$ and any pair of antimorphisms $\Psi_1, \Psi_2 \in G$.
In \cite{PeSta1}, for such a group, the number $1$ is called  $G$-distinguishing, since the image of a single letter by an antimorphism from $G$ allows to identify the antimorphism.
For example, the number $1$ is $H$-distinguishing for the group  $H$ used in  Example \ref{ThueMorsefinite}. Also for the dihedral groups  $I_2(m)$  studied in Section \ref{sec:multi}, the number $1$ is   $I_2(m)$-distinguishing.

If an infinite word $\uu$ is closed under a group $G$ and $1$ is $G$-distinguishing, then
\begin{equation} \label{eq:Gnerovnost}
\Delta \C_{\uu} (n) + \# G \ \ \geq   \sum\limits_{ \substack{ \Psi \in G^{(2)}  }}\Bigl(\P^{\Psi}_{\uu}(n) + \P^{\Psi}_{\uu}(n+1)\Bigr) \qquad   \text{for all \ \ } n \in \N, n\geq 1,
\end{equation}
where   $G^{(2)}$ denotes the set of all involutory antimorphisms from $G$, see \cite{PeSta1}.
Clearly, if $G$ is generated by one antimorphism, say  $\Psi$, then $\#G=2$ and $ G^{(2)} =\{\Psi\}$.  The inequality  \eqref{eq:BaMaPe} is the special case of \eqref{eq:Gnerovnost}. Similarly, the inequality \eqref{eq:NerovnostH}  can be obtained from \eqref{eq:Gnerovnost} if we put $G=H = \{\id, R, E, ER\}$.
The following  $G$-analogue of the result   obtained by  Bucci, De Luca, Glen and Zamboni in \cite{BuLuGlZa} for the classical richness is proved in \cite{PeSta2}.
\begin{theorem}\label{rovnost=defect}
Let an infinite word $\uu$ be closed under a group $G$ such that the number $1$ is $G$-distinguishing.
The $G$-defect  $D^G(\uu)$ is zero if and only if   in \eqref{eq:Gnerovnost} the equality holds for each $n\in \N, n\geq 1$.
\end{theorem}
 In \cite{PeSta1} we also introduced the  notion  almost $G$-rich word. A word   $\uu$ closed under a group $G$ is  \emph{almost $G$-rich} if  there exists $N \in \N$ such that  the equality in \eqref{eq:Gnerovnost} takes place for all integers $n\geq N$.
 An infinite word $\uu$ is almost $G$-rich if and only if its $G$-defect
 \[
 D^G(\uu) = \sup \{ D^G(w) \colon w \in \Lu\}
 \]
  is finite.

\begin{remark}\label{BezUR}
In fact, in \cite{PeSta2} the last statement is shown only for uniformly recurrent words.
However, one can use the same argument we applied in proof of Theorem 2 in \cite{BaPeSta5} and show that $ D^G(\uu)$ is finite  if and only if  in  \eqref{eq:Gnerovnost} the equality takes place from some $N$ on.
\end{remark}


\section{Binary words invariant under two involutory antimorphisms} \label{sec:binary}

\subsection{$G$-richness in binary alphabet}\label{spolu}

In this section we suppose $\A = \{0,1\}$. On binary alphabet we
have only two antimorphisms $R$ and $E$. Therefore, only the groups
\[
\{ \id, R\}, \quad \{ \id, E\}, \quad \hbox{and }\quad H=\{ \id, R, E,
ER\},
\]
can be considered when inspecting the defect  $D^G$. Let us start
with examples of $G$-rich and almost $G$-rich words for these
three groups.

\begin{example} {($G=\{ \id, R\}$)}\\
The classical richness has been  studied very intensively and thus
there are known many examples of binary $R$-rich words including
Sturmian words, see \cite{DrJuPi}, Rote Words, see \cite{BlBrLaVu11}, the
period doubling word, see \cite{Ba_phd}, etc.
 Plenty examples of binary almost $R$-rich words can be constructed
 by application of special standard $P$-morphisms to any rich word, see
\cite{GlJuWiZa} for the  definition of standard $P$-morphism and a proof.
\end{example}

\begin{example} {($G=\{ \id, E\}$)}\\
It can be easily seen, or shown using the results of
\cite{BlBrGaLa}, that there exist only two $E$-rich infinite
words, namely the periodic word $\uu=(01)^\omega$ and its shift
$(10)^\omega$. The two mentioned words are also $R$-rich and
$H$-rich as the equalities hold  in  \eqref{eq:BaMaPe} and
\eqref{eq:NerovnostH}
 for  all $n \in
 \mathbb{N}, n\geq 1$.

Examples of infinite words with finite $E$-defect are \emph{$E$-standard words with seed} (see \cite{BuLuLuZa2} for their definition and \cite{PeSta_Milano_IJFCS} for a proof).
This class also includes very simple examples of words with finite $E$-defect: periodic words having the form $w^\omega$ with $w = E(w)$.
One can easily show that in this case $D^E(w^\omega) = D^E(w^2)$ (see Corollary 8 in \cite{BrHaNiRe} for $R$-defect, a modification for $E$ is straightforward).
\end{example}

\begin{example}\label{ThueMorse} {($G=H=\{ \id, R, E,
ER\}$)}\\
The only so far known examples of  $H$-rich words are given in
\cite{Sta2011}: they are the generalized Thue--Morse words
$\tt_{b,2}$.

 If $b$ is odd, then
$\tt_{b,2} = (01)^\omega$ and hence  $\tt_{b,2}$ is  also $R$-rich
and $E$-rich.

If $b$ is even, the word is aperiodic and  $D^R(\tt_{b,2}) =
D^E(\tt_{b,2}) = + \infty$.  To prove it for any even $b$ we use
the fact that $\tt_{b,2}$ is a fixed point of the morphism
$\varphi$ determined by
\[
\varphi: \ \ \ 0 \mapsto (01)^{\frac{b}{2}} \quad \text{ and }
\quad 1 \mapsto (10)^{\frac{b}{2}}.
\]

It is readily seen  that the factor $w=(01)^{\frac{b}{2}}$ is
\emph{strong}, i.e., its bilateral order $b(w)$ is positive, specifically $b(w) =1$,  as all four
words $0w1$, $0w0$, $1w1$, and $1w0$ belong to $\Lu$. Moreover $w$
is an $E$-palindrome.  The form of the morphism ensures that
\begin{itemize}
\item  $b(\varphi(v)) =1$ for any strong factor $v\neq
\varepsilon$; \item if $v$ is an  $R$-palindrome, then
$\varphi(v)$ is an $E$-palindrome, \item if $v$ is an
$E$-palindrome, then $\varphi(v)$ is an $R$-palindrome,
\end{itemize}
These properties imply that for any $k\in \mathbb{N}$,  the factor
$ \varphi^{2k} \big ( w \big )$   is an $E$-palindrome and hence
it is not  an $R$-palindrome. Thus there exist infinitely many
non-palindromic bispecial factors with non-zero bilateral order.
Using \Cref{thm:minus1} one may see that $D^R(\tt_{b,2}) = +
\infty$.

To prove that $D^E(\tt_{b,2}) = + \infty$ we may proceed
analogously. The factors $ \varphi^{2k+1} \big ( w \big )$ are
$R$-palindromes but they are  not $E$-palindromes  for all $k >
0$. These factors are bispecial with the same bilateral order $1$.
A modification of \Cref{thm:minus1} for the antimorphism $E$
(which can be found in full generality in \cite{PeSta2},
Proposition 45) gives the result.
\end{example}

Now we look at the question whether a word can be simultaneously
(almost) $G$-rich for two groups on the binary alphabet.
We will discuss the connection  between finiteness of defects $D^R$, $D^E$
and $D^H$.
In what follows we will consider words invariant under $R$ and $E$ simultaneously. First we study the relationship between $R$- and $E$-palindromes.

\begin{lemma} \label{le:pq}
Let $p,q \in \A^*$ be $R$-palindromes such that the word $pq$ is an $E$-palindrome, i.e.,
\begin{equation} \label{le:pq_1}
pq  = E(q)E(p).
\end{equation}
There exist $c \in \A^*$ and $i,j \in \N$ such that $p = c \left( E(c) c \right)^i$ and $q = \left( E(c) c\right)^j E(c)$.
\end{lemma}

\begin{proof}
We will induce on the difference of $|p|$ and $|q|$.
First, suppose that $|p| = |q|$, then \eqref{le:pq_1} implies that $q = E(p)$ and it suffices to set  $c = p$ and $i = j = 0$.

Suppose now that $|p| \neq |q|$.
We can suppose without loss of generality that $|p| < |q|$.
Set $q = q_1q_2$ with $|p| = |q_2|$.
It follows from \eqref{le:pq_1} that $pq_1q_2 = E(q_2)E(q_1)E(p)$, thus $p = E(q_2)$ and $q_1 = E(q_1)$.
Therefore, $q_2$ is a palindrome.
Since $q$ is a palindrome, we have $R(q_1q_2) = R(q_2)R(q_1) = q_1 q_2 = q_2 R(q_1)$.
We get
\begin{equation} \label{le:pq_2}
q_1 q_2 = q_2R(q_1).
\end{equation}
This equation on words, written in general as $xz = zy$, has a
well-known solution:  there exist words $u,v \in \A^*$ and $k \in
\N$ such that $x = uv, y = vu$ and $z = (uv)^ku$. If the word $z$
is palindrome, then the form of $z$ implies that $u$ and $v$ are
palindromes as well. To use the solution of $xz = zx$ to solve
\eqref{le:pq_2}, we set $z = q_2$, $x = q_1$ and $y = R(q_1)$ and
we get the solutions $q_1 = uv = E(uv)$ and $q_2 = (uv)^ku$. Since
$|q_1| = |q| - |p| = |u| + |v|$, it follows that the difference of
$|u|$ and $|v|$ is less than $|q_1| = |q| - |p|$. We apply the
induction hypothesis on the palindromes $u$ and $v$ satisfying
$E(uv) = uv$ and we get that $u = d (E(d)d)^m$ and $v =
(E(d)d)^nE(d)$ for some $d \in \A^*$. Substituting for $p$ and $q$
one can find that it suffices to set $c = E(d)$ and the claim is
proved.
\end{proof}

\begin{corollary}\label{periodaER}
If $p$ and $q$ are palindromes such that $pq = E(pq)$, then there exists $c \in \A^*$ such that $pq = (cE(c))^j$ for some $j \in \N$.
\end{corollary}

\begin{proposition}\label{periodic}If an infinite recurrent word  $\uu$  has  finite $R$-defect
and  finite $E$-defect, then  $\uu$ is periodic with a period conjugate to $rE(r)$, where $r$ is an $R$-palindrome.
\end{proposition}

\begin{proof}
Let $\uu$ be an infinite recurrent word with finite $R$- and $E$-defects.
Using Proposition 5 in \cite{PeSta_Milano_IJFCS}, it follows that $\uu$ is closed under $R$ and $E$ and there exists an integer $h$ such that
\begin{align*}
\Delta \C_\uu(n) +2 & = \P^R(n+1) + \P^R(n) \quad \quad \text{ and} \\
\Delta \C_\uu(n) +2 & = \P^E(n+1) + \P^E(n)
\end{align*}
for all $n \geq h$. Since $\uu$ is also closed under all elements of the group
$H$, combining the two previous equalities  with
\eqref{eq:NerovnostH} we get $0 \geq \Delta \C_\uu(n)$ for all $n
\geq h$, i.e., the word $\uu$ is eventually periodic.  Since  $\uu$
is recurrent and closed under $R$,  the word $\uu$ is purely
periodic, i.e., $\uu =  w^\omega$. As  $\uu$ is closed under $E$,
the word   $E(w)$  is a factor of $ww$. It implies  that
$w=w_1w_2$ with  $E(w_1) = w_1$ and $E(w_2) = w_2$.   As the length of
any $E$-palindrome is even, the concatenation of two
$E$-palindromes is conjugate to an $E$-palindrome, in other words,
the word $w$ is conjugate to an $E$-palindrome, say $v$. Thus
$\uu =  w^\omega = v'v^\omega$ for some $v'$. As $ v^\omega$  has
language closed under $R$ as well, by the same reasoning we have $v=pq$,
where $R(p)=p$ and $R(q)=q$. Applying Corollary \ref{periodaER} we
get $v= pq = (cE(c))^j$ for some $j \in \N$. It is enough  to set
$ r=c$.
\end{proof}

The following proposition treats another combination of two $G$-defects.

\begin{proposition}\label{RrichImplikHrich}
Let $\uu \in \{0,1\}^\N$ be a  word having its language closed under the group $H$ and let
 $\Psi =R$ or  $\Psi =E$.   If $D^\Psi (\uu )$ is finite (resp. zero), then $D^H(\uu )$ is
finite (resp. zero) as well.
\end{proposition}

Before giving a proof of the last proposition, we recall Proposition 4.3 of \cite{BaPeSta3} which will be needed.

\begin{proposition}\label{prop:7}
Let $\uu$ be an infinite word with language closed under reversal.
Suppose that there exists an integer $N$ such that for all $n \geq N$ the equality $\P^R_\uu(n) + \P^R_\uu(n + 1) = \C_\uu(n + 1) - \C_\uu(n) + 2$ holds.
The complete return words of any palindromic factor of length $n \geq N$ are palindromes.
\end{proposition}

\begin{proof}[of \Cref{RrichImplikHrich}]
Let us realize that closedness of $\uu$ under $R$ and $E$
ensures that  the  numbers $\P^E_\uu(n)$ and   $\P^R_\uu(n)$ are even.
Indeed, if $w \in \Lu$ is an $E$-palindrome of length $n$, then $R(w)$ is
an  $E$-palindrome as well, and analogously for $R$-palindromes.

First we consider  $\Psi = R$.  Let us suppose that there exists a positive integer $N$ such that
\begin{align*}
\Delta \C_\uu(n) + 2 & =  \P_{\uu}^R(n) + \P_{\uu}^R(n+1) \quad \quad  \text{for all $n\geq N$ and } \\
 \Delta \C_{\uu}(N) + 4 & >  \P_{\uu}^R(N) +
\P_{\uu}^R(N+1) + \P_{\uu}^E(N) + \P_{\uu}^E(N+1).
\end{align*}
We will show that this assumption leads to a  contradiction.

In particular the assumption yields  the inequality  $2>  \P^E_\uu(N) + \P^E_\uu(N+1)$, which implies  that there is no $E$-palindrome of length at least $N$. Let
$w \in \L(\uu)$ be an $R$-palindrome of length at least $N$.  We say
that a factor $f$ has Property $\pi$ if it satisfies all of the following:
\begin{enumerate}[1)]
\item $w$ occurs in $f$ exactly once,
\item $E(w)$ occurs in $f$ exactly once,
\item $w$ is a suffix or a prefix of  $f$,
\item $E(w)$ is a suffix or a prefix of  $f$.
\end{enumerate}
Let $u$ be a factor with Property $\pi$. Such factor must
exist as $\L(\uu)$ is closed under $E$ and thus
 $E(w) \in \L(\uu)$ as well.  As $w$ is an $R$-palindrome and
$\uu$ is closed under reversal, the factor $R(u)$ has Property $\pi$
as well. Since $ER(w) = E(w)$, we can assume without loss of
generality that $u$ is the factor starting in $w$ and ending in
$E(w)$. Let us look at the  complete return word of $w$, say $p$,
with prefix $u$. The fact that the  equality $\Delta \C_\uu(n) + 2 =
\P_{\uu}^R(n) + \P_{\uu}^R(n+1)$ is valid for all $n\geq N$ implies
according to \Cref{prop:7} that
 the complete return word
$p$ of $w$ is an $R$-palindrome.  Thus  the factor $R(u)$ is a suffix of
$p$. Moreover $p$ contains only two factors (namely $u$ and $R(u)$)
with Property $\pi$.

We have shown for every  factor $u'$ with Property $\pi$ that its closest right neighbor   in
$\uu$ with Property $\pi$ is its mirror
image $R(u')$. Therefore, there exist only two factors with
Property $\pi$, namely $u$ and $R(u)$.

On the other hand, if $u$ has Property $\pi$, then $E(u)$ has
Property $\pi$ as well and thus $E(u) \in \{u, R(u)\}$.  As $E(w)$
is a suffix of $u$, the factor $E(u)$ has a prefix  $ w$.  It
implies that $E(u) = u$ which contradicts the fact that there
is no $E$-palindrome longer than $|w|$.

 We have shown that
 \[
 \Delta \C_\uu(n) + 2 = \P_{\uu}^R(n) + \P_{\uu}^R(n+1) \quad \quad \text{ for
 all } n\geq N\] implies  \[\Delta \C_{\uu}(n) + 4 = \P_{\uu}^R(n) +
\P_{\uu}^R(n+1) + \P_{\uu}^E(n) + \P_{\uu}^E(n+1) \quad \quad \text{ for
 all } n\geq N.
 \]

If $\uu$ is $R$-rich, then $N=1$ and thus $\uu$ is also $H$-rich. If its
defect $D(\uu)$ is finite but nonzero, then $N>1$ and  $\uu$ has
finite $H$-defect.

In the case $\Psi=E$ the proof is analogous.
\end{proof}

\begin{corollary}\label{RoteHrich} Every complementary-symmetric Rote word is
$H$-rich.
\end{corollary}

\begin{proof}
In \cite{BlBrLaVu11}, it is proved that Rote words
are $R$-rich. Since a complementary-symmetric Rote
word is closed under $H$, the previous theorem proves the statement.
\end{proof}

\begin{remark} Let us stress that the reverse implication in
Proposition \ref{RrichImplikHrich} does not hold. As shown in
Example \ref{ThueMorse},  the Thue--Morse word has $D^H(\tt)=0$ , whereas
$D^R(\tt)=D^E(\tt)=\infty$.

According to Proposition \ref{periodic}, the  finiteness of  both
defects $D^E(\uu)$ and $D^R(\uu)$ forces the word $\uu$ to be periodic. The Rote
words illustrate that there exist aperiodic words with finite
$D^H(\uu)$ and $D^R(\uu)$.

\end{remark}

\subsection{The mapping $\SS$ on binary words}

In this section we introduce and study the basic properties of the mapping $\SS: \A^*\setminus\{\varepsilon\} \to \A^*$ that is given by
\[
\SS(u_0 \cdots u_n) = v_1 \cdots v_n,
\hbox{ \ \ where \ \ } v_i = (u_{i-1} + u_i) \bmod{2}\ \ \hbox{ for
\ } i = 1,\ldots,n.
\]
 In particular, $\SS(a) = \varepsilon$ for every
$a \in \A $. The following list contains some elementary properties
of $\SS$.
\begin{enumerate}[I.]
\item\label{Skomutace} $SR = RS$,  and \  $SE=SR$.
\item\label{Spreimage}  $S(w)=S(u)$ if  and only if $w=u$ or $w=ER(u)$.
\item\label{Sp1} $S(w)$
is an $R$-palindrome if and only   $w$ is an $R$-palindrome or an
$E$-palindrome.

\begin{proof} Points \ref{Skomutace} and \ref{Spreimage} give
\[
S(w)=R\bigl(S(w)\bigr) \Longleftrightarrow  S(w)=S\bigl(R(w)\bigr)  \Longleftrightarrow w=R(w) \
\hbox{or} \ w = ER\bigl(R(w)\bigr)=E(w).
\]
\end{proof}
\end{enumerate}

The operation $\SS$ is naturally extended to $\A^\N$ by setting
\[
\SS(u_0u_1u_2 \ldots) = v_1v_2 \ldots,
\hbox{ \ \ where \ \ } v_i = (u_{i-1} + u_i) \bmod{2}\ \ \hbox{ for
\ } i \geq 1.
\]
To describe the factor complexity of $S(\uu)$  we study
special factors in
\[
\L\bigl(S(\uu)\bigr)=\{S(v)\colon v \in \L(\uu)\}.
\]

\begin{lemma}\label{Sp3a}
Let $\uu \in \{0,1\}^\N$.
A factor  $S(v)$ is right special in $\L(S(\uu))$ if and only
if one of the following occurs:
\begin{enumerate}[a)]
 \item\label{rightspecA}  $v$ or $ER(v)$ is  right special in $\Lu$,
 \item\label{rightspecB}   $\{v, ER(v) \}\subset \Lu $,  and   $\{va,ER(va)\}
\not \subset \Lu$  for both $a\in \{0,1\}$.
 \end{enumerate}
 \end{lemma}
\begin{proof}
Let $S(v)$ be right special in $\L(S(\uu))$. Then $S(v0)$ and
$S(v1)$ belong to $\L(S(\uu))$. It may happen  that either both $v0$
and $v1$ belong to $\Lu$, which means that $v$ is right special in
$\Lu$, or both $ER(v0)$ and $ER(v1)$ belong to $\Lu$, which means
that $ER(v)$ is right special in $\Lu$.

Otherwise  $v$ and $ER(v)$   are not right special  in $\Lu$, but
necessarily both belong to $\Lu$. Let $va$ and $ER(v)b$  be the
unique right prolongations in $\Lu$ of $v$ and $ER(v)$ respectively.
Since $S(va)$ and $S\bigl(ER(v)b\bigr)$ must be distinct right
prolongations of $S(v)=S\bigl(ER(v)\bigr)$, we have $a\neq ER(b)$,
i.e., $a=b$. Since $ER(v)$ has a unique extension to the right $ER(v)a$,  we get $ER(v)(1-a) = ER(va) \notin \Lu$.
\end{proof}

\begin{lemma}\label{lem:UniformRec} Let $\uu \in \{0,1\}\mathbb{^N}$.
The word $\uu$ is uniformly recurrent if and only if  $S(\uu)$ is
uniformly recurrent.
\end{lemma}
\begin{proof}
$(\Rightarrow)$:  Let $w$ be a factor of $S(\uu)$.  Then $w=S(v)$ for some $v \in \L(\uu)$.
The gaps between the neighboring occurrences of $v$ in $\uu$  are
bounded by some constant.
The gaps between the occurrences of $w$ in  $S(\uu)$ are bounded by the same constant.

 $(\Leftarrow)$: Let $v$ be a factor of $\uu$.  Then $w=S(v)$ is
 a  factor of $S(\uu)$ and the gaps between the occurrences of $w$ are bounded, say by $K$.
 If $v$ is the only factor of $\uu$ such that
 $w=S(v)$, i.e., $v$ is the only preimage of $w$ by $S$ in $\uu$,  then the occurrences of $v$ in $\uu$ are bounded by $K$ as well.
 Let us suppose that $w$ has more preimages in $\uu$. According to Property II,
 there are only two preimages of $w$, namely  $v$ and $ER(v)$. Let
 $f$ be a factor of $\uu$ such that $v$ is a prefix of $f$ and
 $ER(v)$ is a suffix of $f$ and $v$ and $ER(v)$ occur in $f$ only
 once. Then $S(f)$ is a complete return word of $w = S(v)=S(ER(v))$.
 As $S(\uu)$ is uniformly recurrent, the gaps between the occurrences of the factor $S(f)$  are bounded, say by $C$.
Both possible preimages of $S(f)$ in $\uu$, namely $f$ and $ER(f)$,
contain
 $v$ either as its prefix or its suffix. Thus the   gaps between the occurrences
of $v$ in $\uu$ are bounded by $C$ as well.
\end{proof}

\begin{lemma} \label{le:RorE} Let $\uu \in \{0,1\}\mathbb{^N}$.
If $S(\uu)$ is closed under $R$, then $\uu$ is closed under $R$ or under $E$.
\end{lemma}

\begin{proof}
Let $v$ be a prefix of $\uu$.
The word $S(v)$  is a factor of $S(\uu)$. According to
the assumption, $RS(v) = SR(v)$ belongs to $\L\bigl(S(\uu)\bigr)$ as
well. Due to Property II, either $R(v)$ or $E\bigl(RR(v)\bigr) = E(v)$
belong to $\Lu$. Thus
\begin{enumerate}[a)]
\item either there exist infinitely many prefixes $v  \in \Lu $ such that $R(v)
\in \Lu$;
\item or there exist infinitely many prefixes $v  \in \Lu $ such that $E(v) \in
\Lu$.
\end{enumerate}


Let us suppose that a) happens.
For any $w \in \Lu$ we may find a prefix $v$ such that $R(v) \in \Lu$ and $w$ is a factor of $v$.
Thus, $R(w) \in \Lu$ and we can conclude that $\uu$ is closed under $R$.

The case b) is analogous.
\end{proof}

\begin{example}\label{ex:PeriodDoubling}  The period doubling word
is the fixed point of the primitive morphism
\[
\varphi_{PD}:  0\mapsto 11 \quad  \text{ and }  \quad   1\mapsto 10.
\]
Thus
\[
\uu_{PD} = 101 110 101 011 101 110 111 010 10 \ldots .
\]
It is well-known that  the period doubling word is the image of the
Thue--Morse word  $\tt$ by $S$.

The word  $ \uu_{PD} = S(\tt)$  is closed under $R$, the
word $\tt$ is closed under $R$ and $E$. It illustrates
that in the previous lemma the simultaneous closedness  under $R$
and $E$ is not excluded.
\end{example}

The previous lemma  guarantees  that  $\uu$ is closed at least under
one of the antimorphisms $E$ and $R$.  We now focus on a property of
$S(\uu)$ that ensures that $\uu$ is closed under both of them.

\begin{lemma}
Let  $\vv = \SS(\uu) \in \{0,1\}^\N$. The language $\Lu$ contains
infinitely many $E$-palindromes and $R$-palindromes if and only if
$\Lv$ contains infinitely many $R$-palindromes centered at the
letter $1$ and infinitely many $R$-palindromes not centered at the
letter $1$.
\end{lemma}

\begin{proof}
Let $u$ be a finite non-empty word  and let $v = \S(u)$.
It suffices to realize the following:
\begin{enumerate}
\item $u$ is an $E$-palindrome if and only if $v$ is an $R$-palindrome centered at the letter $1$;
\item $u$ is an $R$-palindrome of even length if and only if $v$ is an $R$-palindrome centered at the letter $0$;
\item $u$ is an $R$-palindrome of odd length if and only if $v$ is an $R$-palindrome of even length, i.e., centered at $\varepsilon$.
\end{enumerate}
\end{proof}

\begin{corollary}\label{centered}
Let  $\vv = \SS(\uu) \in \{0,1\}^\N$ be uniformly recurrent.
If $\Lv$ contains infinitely many $R$-palindromes centered at the letter $1$ and infinitely many $R$-palindromes not centered at the letter $1$,
then $\uu$ is closed under all elements of $H$.
\end{corollary}
\begin{proof}
The previous lemma implies that $\Lu$ contains infinitely many $E$-palindromes and $R$-palindromes.
Let $w \in \Lu$.
Since $\Lu$ contains $R$-palindromes of arbitrary length and $\uu$ is uniformly recurrent by \Cref{lem:UniformRec}, the factor $w$ is a factor of an $R$-palindromic factor of $\uu$, thus $R(w)$ also occurs in $\uu$.
Analogously, $E(w)$ is factor of an $E$-palindromic factor and thus $E(w) \in \Lu$.
\end{proof}

An example of application of the last corollary are Sturmian
words. It is known that they contain infinitely many
$R$-palindromes centered at $1$ and $0$, which implies that their
preimages by $\SS$, namely the complementary-symmetric Rote words,
have their language closed under $H$.

Another example is the period doubling word defined in
\Cref{ex:PeriodDoubling}. One can easily see can that
given an $R$-palindrome $w$ centered at $x \in \{0,1\}$, the word
$\varphi_{PD}(w)1$ is also an $R$-palindrome centered at $1-x$.
Therefore, the period doubling word satisfies the assumptions of
the corollary and it follows that the language of one of its
preimage by $\SS$, namely the Thue--Morse word, is closed under
$H$.

The following definition is inspired by \cite{BaBuLuHlPu}.
Given a finite word $v$ and a letter $a$, the notation $|w|_a$ stands for the number of occurrences of the letter $a$ in $v$.

\begin{definition} \label{def:WELLDOC}
We say that an infinite word $\vv \in \{0,1\}^\N$ has \emph{well
distributed occurrences modulo $2$} (denoted WELLDOC(2)) if  for every
factor $w \in \Lv$ we have
\[
\Big \{ \big (|v|_0, |v|_1 \big) \bmod{2} \colon vw \text{ is a
prefix of } \vv \Big \} = \Z_2^2.
\]
\end{definition}

\begin{proposition} \label{prop:wdo2}
Let $\vv \in \{0,1\}^\N$ have WELLDOC(2) and be closed under
reversal. If $\uu$ is a word such that $\vv = \SS(\uu)$, then $\uu$
is closed under all elements of $H$.
\end{proposition}

\begin{proof}
Denote $\vv = v_1v_2 \ldots$ and $\uu = u_0u_1 \ldots$. Since
$\SS(\uu) = \vv$, it follows that $u_0 + u_1 = v_1 \mod{2}$, $u_1 +
u_2 = v_2 \mod{2}$, $\ldots$ Summing first $k$ equations we get
\[
u_k = u_0 + \sum_{i=1}^k v_i.
\]
It follows that
\[
u_{k+j} = u_{k-1} + \sum_{i=k}^{k+j}v_i
\]
for all $k > 0$ and $j \in \N$. Suppose that $s$ and $\ell$ are two
distinct occurrences of a factor $f \in \Lv$ of length $n$. We have
$\sum_{i = s}^{s+j} v_i = \sum_{i = \ell}^{\ell+j} v_i$ for all $j
\in \{0,\ldots, n-1\}$. If $u_{s-1} = u_{\ell-1}$, then we have
$u_{s-1} \cdots u_{s+n} = u_{\ell-1} \cdots u_{\ell+n}$. On the
other hand if $u_{s-1} \neq u_{\ell-1}$, then $u_{s-1} \cdots
u_{s+n} = ER(u_{\ell-1} \cdots u_{\ell+n})$. Note that $u_{s-1} =
u_0 + \sum_{i=1}^{s-1} v_i$ and $v_1\cdots v_{s-1}f$ is a prefix of
$\vv$, and analogously for the index $\ell$. Since $\vv$ has WELLDOC(2),
we may choose the indices $s$ and $\ell$ such that
\[
\sum_{i=1}^{s-1} v_i = 0 \quad \text{ and } \quad
\sum_{i=1}^{\ell-1} v_i = 1.
\]
It implies that with every factor $w \in \Lu$, the factor $ER(w)$
also occurs in $\uu$.
As $\vv = S(\uu)$ is closed under reversal, \Cref{le:RorE} implies that $\uu$ is closed under $R$ or $E$.
This together with the closedness under $ER$ already implies that $\uu$ is closed under all elements of $H$.
\end{proof}


\subsection{Richness of $\uu$ versus richness of  $S(\uu)$ } \label{sec:richness_vs_richness}

This section is devoted to the study of images and preimages of almost rich words by the mapping $S$.

\begin{theorem} \label{thm:bin1}
Let $\uu \in \{0,1\}^\N$ be  closed
under all elements of $H = \{ \id, E, R, ER\}$.  The word $\uu$ is
$H$-rich (resp. almost $H$-rich) if and only if $\SS(\uu)$ is
$R$-rich (resp. almost $R$-rich).
\end{theorem}

\begin{proof}
Let $\vv = \SS(\uu)$. Since $\uu$ is closed under $ER$,
any factor of $\Lv$ has two preimages in $\Lu$. Moreover,  $v$ is
right special  in $\Lu$ if and only if $ER(v)$ is right special
in $\Lu$  as well. Thus by Lemma \ref{Sp3a}, any right special
factor in $\Lv$ of length $n$ is image of two right special
factors of length $n+1$. According to \eqref{1diffBin} we get
\begin{equation} \label{eq:2r1}
  2\Delta \C_\vv(n) = \Delta \C_\uu(n+1).
\end{equation}

Analogously, $v$ is an $R$- or $E$-palindrome in $\Lu$ if and only
if $ER(v)$ is an $R$- or $E$-palindrome in $\Lu$. According to Property
\ref{Sp1} we have
\[
2 \P^R_{\vv}(n) = \P^R_{\uu}(n+1) + \P^E_{\uu}(n+1).
\]
Thus the equality
\[
\Delta\C_\vv(n) +2 = \P^R_{\vv}(n+1) +
\P^R_{\vv}(n),
\]
 testifying that $\vv$ is $R$-rich, holds if and only if the equality
\[\Delta \C_\uu(n+1) +4
=\P_{\uu}^R(n+1) + \P_{\uu}^E(n+1) + \P_{\uu}^R(n+2) +
\P_{\uu}^E(n+2),\] testifying that $\uu$ is $H$-rich, is satisfied.
\end{proof}

As already noted above, the word $\tt_{b,2}$ is $H$-rich.
Thus, using the last theorem with $\uu = \tt_{b,2}$ and \eqref{eq:2r1} together with the equality
\[
\C_\ww(n) = 1 + \sum_{i=0}^{n-1} \Delta \C_\ww(i)
\]
valid for any infinite word $\ww$, we obtain the following corollary:

\begin{corollary} \label{coro:tb2_rich}
For every integer $b$ greater than $1$ the word $\vv = \SS(\tt_{b,2})$ is $R$-rich.
Its factor complexity satisfies
\[
\C_\vv(n) = \frac{1}{2} (\C_{\tt_{b,2}}(n) -1 ).
\]
\end{corollary}

Using the factor complexity of the word $\tt_{b,2}$ described in \cite{Sta2011}, one can see that the binary $R$-rich word $\SS(\tt_{b,2})$ is not Sturmian.

\begin{remark}
Since a complementary-symmetric Rote word $\uu$ is closed under
all elements of $H$ and $\SS(\uu)$ is Sturmian, which is $R$-rich,
Theorem \ref{thm:bin1} provides an alternative proof of \Cref{RoteHrich} without exploiting the result that every
Rote word $\uu$ is $R$-rich.
\end{remark}

\begin{theorem}\label{zustavaRich}
Let $\uu \in \{0,1\}^\N$ be a uniformly recurrent word. If $\uu$ is almost $R$-rich, then  the word $\SS(\uu)$ is almost $R$-rich.
\end{theorem}

\begin{proof}  Since $\uu$ is almost rich, its language  contains infinitely many palindromes. This fact for uniformly recurrent words implies that $\uu$ is closed under reversal.
If $\uu$ is closed under $E$ as well, then according to
Proposition \ref{RrichImplikHrich} the word $\uu$ is almost
$H$-rich, and the  claim follows from \Cref{thm:bin1}.

It is enough to consider $\uu$ that is not closed
under $E$. We will show that the set $\{ w \in \Lu \colon ER(w) \in
\Lu \}$ is finite. Assume the opposite. Let $v$ be a factor of
length $n$. As $\uu$ is uniformly recurrent there exists a number
$r(n)$ such that any factor of $\uu$ longer than $r(n)$ contains all
factors of length $n$. Since $\{ w \in \Lu \colon ER(w) \in \Lu \}$
is not finite, there exists $w$ belonging to this set and being longer than $r(n)$. And thus the
factor $v$ of length $n$   occurs in $w$ and $ER(v)$ occurs in
$ER(w)$. Since both $w$ and $ER(w)$ belong to $\Lu$, the factor $v$
and $ER(v)$ belongs to $\Lu$ as well --- a contradiction with
assumption that $\uu$ is not closed under $E$.

Let $N$ be the maximal length of an element of the finite set   $\{
w \in \Lu \colon ER(w) \in \Lu \}$. Any factor of $S(\uu)$ longer
than $N$ has unique preimage in $\Lu$.  According to Lemma
\ref{Sp3a}
 for any
$n
>N$ there is one-to-one correspondence between $\L_n(S(\uu))$
and $\L_{n+1}(\uu)$.
 Thus,
\begin{equation} \label{prfdeltac}
\Delta \C_\uu(n+1) = \Delta\C_{\SS(\uu)}(n) \quad \quad \text{ for all } n >
N.
\end{equation}

Moreover, there exists no $E$-palindrome of length $n>N$ and thus
we have one-to-one correspondence between   the set of all
$R$-palindromes in $\L(S(\uu))$  of length $n$ and the set of all
$R$-palindromes in $\Lu$ of length $n+1$. It gives

\begin{equation} \label{prfp}
\P^R_\uu(n+1) = \P^R_{\SS(\uu)}(n) \quad \quad \text{ for all } n > N.
\end{equation}

Since $\uu$ has language closed under reversal and is almost rich using \eqref{DefectFormula} there exists a constant $M$ such that
\[
\Delta \C_\uu(n) + 2 = \P_\uu(n) + \P_\uu(n+1) \quad \text{ for all } n \geq M.
\]
This equality and equalities \eqref{prfdeltac} and \eqref{prfp} imply
\[
\Delta \C_{\SS(\uu)}(n) + 2 = \P_{\SS(\uu)}(n) + \P_{\SS(\uu)}(n+1) \quad \text{ for all } n > \max \{N,M\}.
\]
It follows that the word $\SS(\uu)$ is almost $R$-rich.
\end{proof}

\begin{corollary}\label{toSturm}
If $\uu$ is a Sturmian word, then $\SS^k(\uu)$ is almost $R$-rich for all $k > 0$.
\end{corollary}

We add two more examples related to images  (Example \ref{AsiJo}) and preimages (Example \ref{AsiJoJednou}) of words constructed by iterated operation $S$.
However, we do not give any proofs of their properties and we just state them as hypotheses given by computer evidence.

\begin{example}\label{AsiJo}
As stated in \Cref{coro:tb2_rich}, the word $\SS(\tt_{b,2})$ is $R$-rich.
\Cref{zustavaRich} then implies that $\SS^k(\tt_{b,2})$ is almost $R$-rich for all $k > 0$.
Our computer experiments suggest that in this case the word $\SS^k(\tt_{b,2})$ is in fact $R$-rich.
\end{example}


As we have already mentioned, the list of known  $H$-rich words is very
modest: complementary-symmetric Rote words and generalized binary
Thue--Morse words $\tt_{b,2}$. \Cref{thm:bin1},  \Cref{prop:wdo2}
and Corollary \ref{centered} give us a recipe  for construction of
an (almost)  $H$-rich word:  take a binary (almost)  $R$-rich word
with property WELLDOC(2) or a binary (almost) $R$-rich word with
 suitable structure of palindromes and find its preimage by
the operation $S$. The complementary-symmetric Rote words were
obtained by this procedure applied to the Sturmian words. The
Thue--Morse word $\tt=\tt_{2,2}$ can be   obtained   by this
procedure applied to the period doubling word.

\begin{example}\label{AsiJoJednou}
Let $\uu$ be a Sturmian word.
Let $\uu^{(k)}$ be an infinite word such that $\SS^k(\uu^{(k)}) = \uu$ for all $k \in \N$.
The word $\uu^{(1)}$ is a complementary-symmetric Rote word which is, as already mentioned, $H$-rich and $R$-rich.
According to our computer experiments, so is the word $\uu^{(2)}$.
The word $\uu^{(3)}$ is not $R$-rich, but it is still $H$-rich.
The word $\uu^{(k)}$ for $k > 3$ is not $H$-rich nor $R$-rich.
However, the symmetries of $\uu^{(k)}$ are preserved: $\uu^{(k)}$ is closed under all elements of $H$.
This is witnessed by the following difference of its factor complexity
\[
\Delta \C_{\uu^{(k)}}(n) = 2^{n-1} \quad \text{ for } 0 < n \leq k \quad \quad \text{and} \quad \quad \Delta \C_{\uu^{(k)}}(n) = 2^k \quad \text{ for } n > k
\]
which is suggested by our experiments.
\end{example}



\section{The mapping $\SS$ on multiliteral alphabets}\label{sec:multi}

In this section we study the mapping $\SS$ acting on a larger alphabet  $ \Z_m = \{0, \ldots, m-1 \}$.
The mapping $\SS$ is defined  for every word $w = w_0 \cdots w_n$ with $w_i \in \Z_m$ by
\begin{equation}\label{multiS}
\SS(w_0 w_1 \cdots w_n) = v_1 \cdots v_n,
\end{equation}
where $v_i = (w_{i-1}+w_i) \bmod m$ for every $i \in \{1,\ldots,n\}$.

The alphabet $\Z_m$ allows many finite groups generated by involutory antimorphisms.
We restrict our attention to groups isomorphic to groups of symmetries of a regular polyhedron.
The reason is simple: we have examples of $G$-rich words only for such groups, namely  the generalized Thue--Morse words.
We demonstrate that at least for  these words   the mapping  $\SS$   transforms a $G$-rich word to an almost  $G'$-rich word (cf. Theorems   \ref{thm:bin1} and \ref{zustavaRich}  for an analogue on the binary alphabet).

Let us describe the elements of the mentioned group explicitly.
For all $x \in \Z_m$ denote by $\Psi_{x}$ the antimorphism given by
\[
\Psi_{x}(k)  = x - k \quad \text{for all } k \in \Z_m
\]
and by $\Pi_{x}$ the morphism given by
\[
\Pi_{x}(k)  = x + k \quad \text{for all } k \in \Z_m.
\]
The group $I_2(m)$ is the union of these antimorphisms and morphisms:
\[
I_2(m) = \{ \Psi_x \colon x \in \Z_m \} \cup \{ \Pi_x \colon x \in \Z_m \}.
\]
The definition  of the generalized Thue--Morse words is recalled  in Preliminaries. It is known that the  word  $\tt_{b,m}$ is a fixed point of the morphism $\varphi_{b,m}: \Z_m^* \to \Z_m^* $  defined by
\[\varphi_{b,m}: \quad a\mapsto a(a+1)(a+2)\cdots (a+b-1)\ \quad \text{for all }  a \in \Z_m.\] Let us stress that all operations on letters in this section are taken modulo $m$.
As shown in \cite{Sta2011}, the word $\tt_{b,m}$  is closed under all elements of $I_2(m) $  and moreover  $\tt_{b,m}$ is  $ I_2(m)$-rich.
We will focus on   images of  $\tt_{b,m}$  by $S$ with parameters $b\geq 3$ and $m\geq 3$.  Let $ I_2'(m)$ denote the group generated by  antimorphisms
$\{\Psi_{2y}\colon y \in \Z_m\}$; it can be easily seen that
\begin{equation}\label{Gprime}
I_2'(m) = \{ \Psi_{2x} \colon x \in \Z_m \} \cup \{ \Pi_{2x} \colon x \in \Z_m \}.
\end{equation}
If $m$ is odd, then $ I_2'(m)=I_2(m)$, if $m$ is even, then    $ I_2'(m)$ is isomorphic to $ I_2(\frac{m}{2})$. \\

The aim of this section is to prove the following theorem.

\begin{theorem} \label{theoremH}
Let $m, b \in \Z$ such that $m \geq 3$ and $b \geq 3$.
\begin{enumerate}
\item The word  $\SS(\tt_{b,m})$ is  almost $ I'_2(m)$-rich.
\item If   $m$  or  $b$ is odd,  the word $\SS(\tt_{b,m})$ is $ I'_2(m)$-rich.
\end{enumerate}
\end{theorem}

The first part of  Theorem  \ref{theoremH} is a direct consequence  of  Proposition \ref{propH}, the second part  follows from  Lemma \ref{stacin12}  and the description of factors of     $\SS(\tt_{b,m})$  up to the length 3 presented at the end of this section.

\begin{example}  Let us   consider the word $\tt_{4,4}$.  It starts with 0 and it is  a fixed point of the morphism
\[\varphi_{4,4}: \quad 0\mapsto 0123, \quad 1\mapsto 1230,\quad 2\mapsto 2301\  \ \text{and } \  \   3\mapsto 3012 .\]

\centerline{$\text{Thus }\quad \tt_{4,4}  = 01231230230130121230230130120123230130120123  \ldots
$}

\centerline{$\text{and} \quad \SS(\tt_{4,4})  = 1310313213103133313213103132131113103132131 \ldots
$}
Now   we   consider the word $\tt_{3,4}$.  Its fixing morphism is
\[\varphi_{3,4}: \quad 0\mapsto 012, \quad 1\mapsto 123,\quad 2\mapsto 230\  \ \text{and } \  \   3\mapsto 301 .\]

\centerline{$\text{Thus }\quad \tt_{3,4}  = 01212323012323030123030101212323030123030101 \ldots
$}
\centerline{$\text{and} \quad \SS(\tt_{3,4})  = 13331113131113331313331113331113331313331113  \ldots .
$}
\end{example}

We start with a  list of observations concerning properties of the mapping $\SS$.
To deduce some of the observations we exploit a peculiar property of generalized Thue--Morse words.
The form of morphism $\varphi_{b,m }$ forces  the language of   $\uu = \tt_{b,m}$ to have the following  property
\begin{equation}\label{posobe}
   u_0u_{1}u_{2}u_{3} \in \Lu  \ \ \Rightarrow \ \   u_i-u_{i-1} = 1\ \ \text{  for at least two indices $i\in \{1,2,3\}$}.
\end{equation}

\begin{enumerate}[(A)]

\item \label{propm_0}  $\SS\Psi_{y} = \Psi_{2y}\SS$ for any $y\in \Z_m$. If $m$ is even, then    $\SS\Psi_{y} = \SS\Psi_{y+\frac{m}{2}}$ for any $y\in \Z_m$.
\item \label{propm_1} If $\SS(u_0 \cdots u_n) = \SS( v_0 \cdots v_n )$ with $u_i, v_j \in \Z_m$, then there exists $ x \in \Z_m$ such that
\[
v_0 \cdots v_n = (u_0 + x) (u_1 - x) \cdots (u_n + (-1)^n x).
\]
\item\label{propm_2}  Let $\uu$ be  closed under all elements of  $I_2(m)$ and
satisfy \eqref{posobe}. Consider  $w= \SS(v)$ for some $v  =
v_0v_1\cdots v_n\in\L(\uu)$  with $n\geq 3$.

If $m$ is odd, then $v$  is the only preimage of $w$  by  $S$ in $\uu$.

 If $m$ is even, then $w$ has  exactly two  preimages by $S$ in $\uu$, namely $ v_0v_1\cdots v_n$ and $(v_0+\frac{m}{2})(v_1+\frac{m}{2}) \cdots (v_n+\frac{m}{2})$.

\begin{proof}  Let $\SS(u) = \SS( v )$     for  a factor $u = u_0u_1 \cdots u_n \in \Lu$.
As $n \geq 3$, property \eqref{posobe} implies that there exists $j \in \{1,2,3\}$ such that $u_j-u_{j-1} = v_j - v_{j-1} = 1$.
From Property \eqref{propm_1}, we obtain $v_j - v_{j-1} = u_{j} - u_{j-1} + (-1)^j 2x$, thus $2x = 0$.
If $m$ is odd, then necessarily $x=0$. If $m$ is even, then also $x =\frac{m}{2}$ satisfies $2x=0$.
As the morphism $\Psi_0\Psi_{\frac{m}{2}}$ maps $a$ to $a+\frac{m}{2}$, the language of $\uu$ is closed under addition of $\frac{m}{2}$ to all letters of any factor of $\uu$, i.e.,    $(v_0+\frac{m}{2})(v_1+\frac{m}{2}) \cdots (v_n+\frac{m}{2}) \in \Lu$.
\end{proof}

\item \label{propm_3} Let $\uu$ be  closed under all elements of  $I_2(m)$ and
satisfy \eqref{posobe}. Consider $u \in \Lu$.
\begin{enumerate}
\item \label{propm_3_1} If $u$ is a $\Psi_y$-palindrome, then $\SS(u)$  is a $\Psi_{2y}$-palindrome.
\item \label{propm_3_2} If $\SS(u)$ is a $\Psi_{2y}$-palindrome with $|\SS(u)|\geq 3$ and $m$ is odd, then  $u$ is a $\Psi_y$-palindrome.
\item \label{propm_3_3} If $\SS(u)$ is a $\Psi_{2y}$-palindrome with $|\SS(u)|\geq 3$ and $m$ is even, then  $u$ is a $\Psi$-palindrome for $\Psi = \Psi_y$  and  $\Psi = \Psi_{y+
\frac{m}{2}}$.
\end{enumerate}
\begin{proof}
\eqref{propm_3_1}:
Applying $\SS$ to   $u=\Psi_y(u)$ and using  \eqref{propm_0}, one has  $\SS(u) = \SS(\Psi_y(u)) =\Psi_{2y}(\SS(u))$, i.e., $\SS(u)$ is a $\Psi_{2y}$-palindrome.

\eqref{propm_3_2} and \eqref{propm_3_3}:  Using Property
\eqref{propm_0}, we obtain   $\SS(u) =\Psi_{2y}(\SS(u)) =
\SS(\Psi_y(u)) $.
As $|\SS(u)|\geq 3$ implies $|u| \geq 4$ and \eqref{posobe} is satisfied, we may apply Property \eqref{propm_2}.
For odd $m$, it implies $\Psi_y(u)=u$ as we want to show.  For even $m$, we have
also the second possibility
$u=\Psi_y(\Psi_0\Psi_{\frac{m}{2}}(u))$. It is easy to check that
$ \Psi_y\Psi_0\Psi_{\frac{m}{2}} =  \Psi_{y+ \frac{m}{2}}$.
\end{proof}
\end{enumerate}

To prove \Cref{theoremH} we use the notion of complete $G$-return
word of an orbit, as introduced in  \cite{PeSta2}. Let us recall
that the orbit $[w]$ of a factor $w \in \L(\uu)$ is defined by
\eqref{orbit}.

 A factor $v \in \L(\uu)$  is a \emph{complete $G$-return word of $[w]$ in $\uu$} if
\begin{itemize}
\item  $|v| >|w|$,
\item a prefix and a suffix of $v$ belong to $[w]$, and
\item $v$ contains no other elements of $[w]$.
\end{itemize}

\begin{thm}[\cite{PeSta2}] \label{CharakteristikaReturn}
If $\uu$ is an infinite word closed under all elements of $G$, then
\begin{enumerate}
\item $\uu$ is $G$-rich if and only if for all $w \in \Lu$ every complete $G$-return word of $[w]$ is a $G$-palindrome.
\item $\uu$ is almost $G$-rich if and only if there exists and integer $N$ such that for all $w \in \Lu$ longer than $N$ every complete $G$-return word of $[w]$ is a $G$-palindrome.
\end{enumerate}
\end{thm}

Using the previous theorem, we can easily prove the  proposition which directly implies the validity of the first part of   \Cref{theoremH} because  the generalized Thue--Morse words satisfy its assumption.

\begin{proposition}\label{propH} Let $\mathcal{A}=\Z_m$  and let  $\uu \in  \A^\N$  be closed under all elements of $I_2{(m)}$.
If $\Lu$  satisfies \eqref{posobe} and $\uu$  is $I_2(m)$-rich, then
the word $\SS(\uu)$ is almost $I'_2(m)$-rich.
\end{proposition}

\begin{proof}
To ease the notation  put  $G=I_2(m)$ and $G'=I_2'(m) $.
 If $m$ is odd then $G=G'$. Otherwise, $\#G= 2\#G'$.

As $\uu$ satisfies \eqref{posobe}, we can apply Property  \eqref{propm_3} to each  palindrome $S(u)$ in  $\SS(\uu)$  with
 length $|S(u)|\geq 3$.
Property \eqref{propm_3}  implies  that $u\in \Lu$ is a $G$-palindrome if and only if $S(u)$ is a $G'$-palindrome in $\L(\SS(\uu))$.
Let $\SS(u)$ be a $G'$-palindrome in $\SS(\uu)$  and $\SS(v)$ be a complete $G'$-return word of  $[\SS(u)]$ in $\SS(\uu)$.
Then  $v$ is a complete $G$-return word   in $\uu$ of $[u]$. The word $\uu$ is $G$-rich and due to Theorem \ref{CharakteristikaReturn}, the factor  $v$ is a $G$-palindrome. According to Property \eqref{propm_3}, the  complete $G'$-return word $S(v)$ is a $G'$-palindrome.  Thus  $\SS(\uu)$ is almost  $G'$-rich.
\end{proof}

In the remaining part of this section we focus on    $G'$-richness of $\SS(\tt_{b,m})$  in the case when  $m$  or $b$ is odd.
\begin{lemma}\label{stacin12} Let $G'=I_2'(m) $  and $\vv = \SS(\tt_{b,m})$.
If  for $n = 1$ and $n=2$ the equality
\begin{equation} \label{eq:Grovnost}
\Delta \C_{\vv} (n) + \# G' \ \ =  \sum\limits_{ \substack{ \Psi \in G' \\ \Psi \text{ is an antimorphism} }}\Bigl(\P^{\Psi}_{\vv}(n) + \P^{\Psi}_{\vv}(n+1)\Bigr) \qquad
\end{equation}
 holds, then
$\vv = \SS(\tt_{b,m})$ is $G'$-rich.
\end{lemma}

 \begin{proof} We combine the results of \cite{PeSta2}.

It is easy to see that  $\Psi_{2x}(a) \neq \Psi_{2y}(a)$  for any $a \in \Z_m$ and  any pair of distinct antimorphisms $\Psi_{2x}$ and $\Psi_{2y}$ from  $G'$.  This property  guarantees that the number $1$ is $G'$-distinguishing in the sense of  Definition 7 of \cite{PeSta2}.  In the proof of  Proposition \ref{propH} we verify that   any complete $G'$-return word of $[w]$  in $\vv$ is a $G'$-palindrome for each     $G'$-palindrome $w\in \Lv$ of length at least $3$.
As $1$ is $G'$-distinguishing, Lemma 28 of \cite{PeSta2} says that  this fact  implies equality \eqref{eq:Grovnost} for all $n \geq 3$.
 According to Proposition 42 of \cite{PeSta2} the word $\vv$ is $G'$-rich if  the equality in \eqref{eq:Grovnost}  holds for each $n \in \N, n\geq 1$.
\end{proof}

 In the case $m$ or $b$ odd and   $n=1$ and $n=2$, we will confirm the equality \eqref{eq:Grovnost}  in the next lemma.
To ease the notation we  put  \[F(n) =  \sum\limits_{ \substack{ \Psi \in G' \\ \Psi \text{ is an antimorphism} }}\P^{\Psi}_{\vv}(n) \]
and $\rho: \Z_m \to \Z_m$ denotes the  permutation
\[\rho(k) = k + b - 1 = \Pi_{b-1}(k).
\]
We will use the following statements  from \cite{Sta2011}. Let $q$
denote the order of $\rho$, i.\,e., the least positive integer
$q$ such that $q(b-1) \equiv 0 \pmod m$. The factors of length $2$
of $\tt_{b,m}$ are
\[
\L_2(\tt_{b, m}) = \{ \rho^k(r-1)r \colon r \in \Z_m, 0 \leq k \leq q-1 \}
\]
and of length $3$
\[
\L_3(\tt_{b,m}) = \{ \rho^k(r-1) r(r+1) \colon r \in \Z_m, 0 \leq k \leq q-1 \} \cup \{ (r-1)r \rho^{-k}(r+1) \colon r \in \Z_m, 0 \leq k \leq q-1 \}.
\]
It is easy to deduce the set of all factors of length $4$:
\begin{align*}
\L_4(\tt_{b,m})  = & \{ \rho^k(r-1) r(r+1)(r+2) \colon r \in \Z_m, 0 \leq k \leq q-1 \} \\
& \cup \{ (r-2) (r-1)r \rho^{-k}(r+1) \colon r \in \Z_m, 0 \leq k \leq q-1 \} \\
& \cup \{ (r-1) r \rho^{-k}(r+1) (\rho^{-k}(r+1)+1)  \colon r \in \Z_m, 0 \leq k \leq q-1  \}.
\end{align*}

\renewcommand*{\arraystretch}{1.4}

\begin{lemma} \label{gtm-C}
Let $\vv = \SS(\tt_{b,m})$.
The numbers of factors of $\vv$ of length $n$ satisfy the following:
\begin{center}
\begin{tabular}{c|c|c|c}
$n$ & $m$ odd & $m$ even, $b$ odd & $m$ even, $b$ even \\ \hline
$1$ & $m$ & $\frac{m}{2}$ & $m$  \\ \hline
$2$ & $qm$ & $\frac{qm}{2}$ & $\frac{3qm}{4}$  \\ \hline
$3$ & $3qm - 2m$ & $\frac{3qm}{2} - m$ & $\frac{3qm}{2} - m$  \\
\end{tabular}
\end{center}
\end{lemma}

\begin{proof}
$n = 1$ \\
It follows from the form of $\L_2(\tt_{b,m})$ that $\L_1(\vv) = \{ \rho^k(r-1) + r \,:\, r \in \Z_m, 0 \leq k \leq q-1 \}$.
We have $\rho^k(r-1) + r = r - 1 + k(b - 1) + r = 2r - 1 + k(b-1)$

If $m$ is odd, then we have directly that $\L_1(\vv) = \Z_m$ (for $k = 0$).

If $m$ is even and $b$ is odd, then $2r - 1 + k(b-1)$ is odd for
every $r$ and $k$, and thus $\L_1(\vv) = \{ 2i + 1 \,:\, 0 \leq i
< \frac{m}{2} \}$.

If $m$ is even and $b$ is even, then for $k = 0$ the number $2r - 1$ is odd, and for $k = 1$ the number $2r - 1 + (b-1)$ is even, and we have $\L_1(\vv) = \Z_m$.

\vspace{\baselineskip}
$n = 2$ \\
The structure of $\L_3(\tt_{b,m})$ implies that the factors of $\vv$ of length $2$ are of the following forms:
\begin{enumerate}
\item \label{L2_forma1} $(\rho^k(r-1)+r)(2r+1) = (2r-1+k(b-1))(2r+1)$ for $r \in \Z_m$ and $0 \leq k < q$, and
\item \label{L2_forma2} $(2r'-1)(\rho^{k'}(r'+1)+r') = (2r'-1)(2r'+1+k'(b-1))$ for $r' \in \Z_m$ and $0 \leq k' < q$.
\end{enumerate}

First, let us see the number of factors of type \ref{L2_forma1}.
Fix $r \in \Z_m$.
Suppose $2r+1 = 2\tilde{r}+1$ for some $\tilde{r}$.
This equation has $1$ solution for $m$ odd and $2$ solutions for $m$ even.
It is easy to see that if $m$ is odd, there are $qm$ distinct factors of type \ref{L2_forma1} and if $m$ is even their number is $\frac{qm}{2}$.
The counts for the second type are exactly the same.

Let us now look how the two types of factors overlap.
Fix $r$ and $k$ and suppose
\begin{align*}
2r-1+k(b-1) & =  2r'-1, \\
2r+1 & =  2r'+1+k'(b-1).
\end{align*}
It follows that $2r' = 2r + k (b-1)$ which may not have a solution only if $m$ is even and $b$ is even, otherwise it has a solution and the two types overlap completely.
If $m$ is even and $b$ is even, the two types overlap only if $k$ is even.
Thus, they overlap in $\frac{qm}{4}$ cases.

\vspace{\baselineskip}
$n = 3$ \\
It follows from $\L_4(\tt_{b,m})$ that there are the following $3$ types of factors of length $3$:
\begin{enumerate}
\item \label{ww1p1} $(\rho^k(r-1) + r)(2r+1)(2r+3)$ for $r \in \Z_m$, $0 \leq k < q$;
\item \label{ww1p2} $(2r'-3)(2r'-1)(r' + \rho^{-k'}(r'+1))$ for $r' \in \Z_m$, $0 \leq k' < q$;
\item \label{ww1p3} $(2r''-1)(r'' + \rho^{-k''}(r''+1))(\rho^{-k''}(r''+1) + \rho^{-k''}(r''+1)+1))$ for $r'' \in \Z_m$, $0 \leq k'' < q$.
\end{enumerate}

Analogously to the previous case $n = 2$, it can be shown that there are $qm$ distinct factors of each type if $m$ is odd, and $\frac{qm}{2}$ distinct factors of each type if $m$ is even.

When investigating the common factors of each type, one can show that each pair has $m$ common factors if $m$ is odd, and $\frac{m}{2}$ common factors otherwise.
Overall, we find that $\# \L_3(\vv) = 3qm - 2m$ if $m$ is odd and $\# \L_3(\vv) = \frac{3qm}{2} - m$ if $m$ is even.
\end{proof}

\begin{lemma} \label{gtm-P}
Let $G' = I'_2(m)$  and $\vv = \SS(\tt_{b,m})$.
The values of $F(n)$  for $n \in \{1,2,3\}$ are as follows:
\begin{center}
\begin{tabular}{c|c|c|c}
$n$ & $m$ odd & $m$ even, $b$ odd & $m$ even, $b$ even \\ \hline
$1$ & $m$ & $\frac{m}{2}$ & $m$  \\ \hline
$2$ & $qm$ & $\frac{qm}{2}$ & $\frac{qm}{4}$  \\ \hline
$3$ & $qm$ & $\frac{qm}{2}$ & $\frac{qm}{2}$  \\
\end{tabular}
\end{center}
\end{lemma}

\begin{proof}
$n = 1$: \\
It is not hard to show that every factor of length $1$ is a $\Psi$-palindrome for a unique $\Psi \in G'$.

\vspace{\baselineskip}

$n = 2$: \\
Suppose that the first type of factor of length $2$ from the proof of \Cref{gtm-C} is a $\Psi_\ell$-palindrome.
We have
\[
2r-1+k(b-1) = \Psi_\ell(2r+1) = \ell - 2r - 1,
\]
which leads to $\ell = 4r -2 + k(b-1)$, i.e., every factor of length $2$ is a $\Psi$-palindrome for some $\Psi \in G'$ except for the case of $m$ even and $b$ even where one may find such $\ell$ only if $k$ is even.
As one can see in the proof of \Cref{gtm-C}, the case of $k$ even is when the two types of factors of length $2$ overlap, thus the total number of $G'$-palindromes is $\frac{qm}{4}$ in this case.

$n = 3$: \\
We will refer to the $3$ types of factors of length $3$ as given in the proof of \Cref{gtm-C} above.
The first two types can be a $\Psi$-palindrome for some $\Psi \in I_2(m)$ if and only if $k = 0$ or $k' = 0$.
This case is also included in the third type for $k'' = 0$, so we need just to check this type.

Suppose that a factor of the third type is a
$\Psi_\ell$-palindrome:
\begin{align*}
2r'' - 1 = \Psi_\ell(\rho^{-k''}(r''+1) + \rho^{-k''}(r''+1)+1) ) = \ell - (2r'' + 3 - 2k''(b-1)), \\
r'' + \rho^{-k''}(r''+1) = 2r'' + 1 - k''(b-1) = \Psi_\ell(r'' + \rho^{-k''}(r''+1)) = \ell - (2r'' + 1 - k''(b-1)).
\end{align*}
Both equalities yield $\ell = 4r'' +2 - 2k''(b-1)$.
Thus, every factor of type $3$ is a $\Psi$-palindrome for some $\Psi \in I_2(m)$.
According to the proof of \Cref{gtm-C}, there are $qm$ such factors if $m$ is odd, and $\frac{qm}{2}$ such factors otherwise.
\end{proof}

\begin{proof}[of the second part of \Cref{theoremH}]
Results of \Cref{gtm-C,gtm-P} can be summarized into the following table:
\begin{center}
\begin{tabular}{c|c|c|c}
$ $ & $m$ odd & $m$ even, $b$ odd & $m$ even, $b$ even \\ \hline
$\Delta \C_{\vv} (1)$ & $(q-1)m$ & $(q-1)\frac{m}{2}$ & $\frac{3qm}{4} - m$  \\ \hline
$F(1)+F(2)$ & $(q+1)m$ & $\frac{(q+1)m}{2}$ & $\frac{qm}{4} + m$  \\ \hline
$\Delta \C_{\vv} (2)$ & $2qm-2m$ & $qm-m$ & $\frac{3qm}{4}-m$  \\ \hline
$F(2)+F(3)$ & $2qm$ & $qm$ & $\frac{3qm}{4}$  \\ \hline
 $\# G'$ &$2 m$&$ m$&$m$
\end{tabular}
\end{center}

Therefore  the assumption  of \Cref{stacin12}  is satisfied in the case when  $m$  or $b$ is odd. Consequently,  $\SS(\tt_{b,m})$    is $G'$-rich.
\end{proof}

\begin{corollary}\label{S4}  Let $b\in \mathbb{N}, b\geq 1$  and $S$ be the  operation defined by \eqref{multiS}  for the alphabet  $\mathbb{Z}_4$. We have
\begin{itemize}
\item   $\SS(\tt_{2b+1,4})$  is  an infinite word over the binary alphabet $\{1,3\}$  and it is $H$-rich  (here $H$ stands for the group generated by the both involutory antimorphisms over the binary alphabet \{1,3\}).

\item   $\SS^2(\tt_{2b+1,4})$  is  an infinite word over the binary alphabet $\{0,2\}$  and it is $R$-rich  ($R$ stands for the reversal mapping over the binary alphabet $\{0,2\}$).

\item   $\SS^k(\tt_{2b+1,4})$  is  an infinite word over the binary alphabet $\{0,2\}$  and it is almost  $R$-rich
for any $k\in \mathbb{N}, k\geq 2$.

\end{itemize}

\end{corollary}

\begin{proof}
The fact that  the alphabet of   $\SS(\tt_{2b+1,4})$  is $\{1,3\}$ is shown in the proof of  Lemma \ref{gtm-C}, where $\L_1(\vv)$ is  described   for any generalized Thue--Morse word $\vv$.
According to the second part of  Theorem   \ref{theoremH}, the word   $\SS(\tt_{2b+1,4})$ is $I'_2(4)$-rich. The group  $I'_2(4)$
defined by \eqref{Gprime} is isomorphic to $ H$.

It is easy to see that  the operation $S$  assigns to any word $\uu\in\{1,3\}^\mathbb{N}$ the word over the alphabet $\{0,2\}$. The second part of the corollary follows from   Theorem \ref{thm:bin1}, the third one from Theorem \ref{zustavaRich}.
\end{proof}

\section{Comments and open questions}

 For infinite words over the binary alphabet $\{0,1\}$,  we illustrated that the  operation $S$   puts into a broader context  the classical richness  (here usually referred to as $R$-richness) and  $H$-richness.  The main open question is which other operation acting on infinite words  behaves analogously. Let us mention here  some  open questions connected with $S$.

\begin{itemize}

\item   The  operation $S$ on $\{0,1\}$ applied  to an almost $R$-rich word  gives  an almost $R$-rich word, see Theorem \ref{zustavaRich}. In particular,  any iteration of $S$  applied to a Sturmian word $\uu$  gives an almost rich word, cf. Corollary \ref{toSturm}. Is  the $R$-defect of $\SS^{k}(\uu)$ zero as suggested by our computer experiments?

\item  The  operation $S$ on $\{0,1\}$ applied to an $H$-rich word   gives  an  $R$-rich word. In particular, $\SS(\tt_{b,2})$ is rich for any generalized Thue--Morse word $\tt_{b,2}$. Our computational experiments suggest
 that $\SS^k(\tt_{b,2})$ is $R$-rich for any $k \in \mathbb{N}, k\geq 1$, see Example \ref{AsiJo}.  Is it true?

\item On the other hand,  any preimage by $S$ of each Sturmian word $\uu$  is $H$-rich and $R$-rich simultaneously, in fact it is a complementary-symmetric Rote word.  Our computer experiments suggest that even the second preimage $\SS^{-2}(\uu)$ is simultaneously $H$- and $R$-rich, whereas    $\SS^{-3}(\uu)$ is only $H$-rich, but not $R$-rich, see Example \ref{AsiJoJednou}.   Is it true?

\end{itemize}
 We have introduced the operation  $S$ over  the alphabet  $\Z_m$ with $m\geq 3$ as well.
But our results on multiliteral alphabet are restricted to special groups and words.

\begin{itemize}
\item  We have considered $G$-richness for $G=I_2(m)$ only.
Proposition \ref{propH}  connects    $I_2(m)$-richness of  $\uu$  and $I'_2(m)$-richness of $\SS(\uu)$   for words $\uu$ satisfying the assumption \eqref{posobe}.   Is  the proposition valid without the assumption?
\item  It would be interesting to study behaviour of ternary episturmian words with respect to operation $S$ on $\Z_3$.
For example, which group of symmetries $G$ has the preimage  of the Tribonacci word by S? Is the preimage  $G$-rich?   Are images of the Tribonacci word by  $S$  still $R$-rich?
\item Corollary \ref{S4} illustrates that the operation $S$ over the alphabet   $\Z_4$ can produce binary almost $R$-rich words as well. What is the $R$-defect of  the words $\SS^k(\tt_{2b+1,4})$?

 \end{itemize}


The last comment  we want to state here  concerns the palindromic closure operator. It is used for construction of standard episturmian words. The construction is governed by a directive sequence of letters $\Delta$.  Any episturmian word $\uu$ is closed under reversal and  $\uu$ is rich  in the classical sense. In \cite{LuLu}, the authors introduced the concept
 of generalized pseudopalindromic closure operator, where  multiple involutory antimorphisms are used. It means that the construction is governed  by  two sequences: a directive sequence of letters  $\Delta$  and  a directive sequence of antimorphisms $\Theta$.   Let us denote the resulting infinite word by $\uu(\Delta, \Theta)$.

 In general, $\uu(\Delta, \Theta)$  is closed under the  group $G$ generated by the involutory antimorphisms occurring infinitely many times in the directive sequence $\Theta$,   but  the word $\uu(\Delta, \Theta)$  need not to be  $G$-rich.  Nevertheless, several examples of $G$-rich words constructed by generalized pseudopalindromic closure operator are already known.
De Luca and de Luca showed that the Thue--Morse  word $\tt =\tt_{2,2}$ can be constructed in this way.
In \cite{JaPeSta1} the generalized Thue--Morse words $\tt_{b,m}$ with the same property are characterized.
The concept of  generalized pseudopalindromic closure on binary alphabet is systematically  studied by Blondin Mass\'e, Paquin, Tremblay and   Vuillon in \cite{BlPaTrVu}. In particular, they proved that any standard complementary-symmetric Rote word can be constructed by using  generalized pseudopalindromic closure operator.  Nevertheless, the question which pairs   $(\Delta, \Theta)$ produce $H$-rich words is open and
 requires  a deeper  study.


\acknowledgements

The first author acknowledges financial support from the Czech Science Foundation grant GA\v CR 13-03538S
and the second author acknowledges financial support from the Czech Science Foundation grant GA\v CR   13-35273P.
The computer experiments were performed using the open-source computer algebra system \texttt{SageMath\textbf{}} \cite{sage_6_10}.
We would like to thank the anonymous referee who provided us with useful hints to improve the article.

\bibliographystyle{abbrvnat}
\IfFileExists{biblio.bib}{\bibliography{biblio}}{\bibliography{../!bibliography/biblio}}

\end{document}